\documentclass{article}
\usepackage{amsmath,amsthm,amsopn,amstext,amscd,amsfonts,amssymb}
\usepackage{natbib}

\usepackage[T2A,OT1]{fontenc}
\usepackage[ot2enc]{inputenc}
\usepackage[russian,english]{babel}

\def \eig {\mathop{\rm eig}\nolimits}
\def \tr {\mathop{\rm tr}\nolimits}
\def \re {\mathop{\rm Re}\nolimits}

\def \Vol {\mathop{\rm Vol}\nolimits}

\def \etr {\mathop{\rm etr}\nolimits}
\def \diag {\mathop{\rm diag}\nolimits}
\def \build#1#2#3{\mathrel{\mathop{#1}\limits^{#2}_{#3}}}

\renewenvironment{abstract}
                 {\vspace{6pt}
                  \begin{center}
                  \begin{minipage}{5in}
                  \centerline{\textbf{Abstract}}
                  \noindent\ignorespaces
                 }
                 {\end{minipage}\end{center}}
\newtheorem{thm}{\textbf{Theorem}}[section]
\newtheorem{cor}{\textbf{Corollary}}[section]
\newtheorem{lem}{\textbf{Lemma}}[section]
\theoremstyle{definition}

\title{\Large \textbf{Matricvariate and matrix multivariate Pearson type II distributions}}
\author{
  \textbf{Jos\'e A. D\'{\i}az-Garc\'{\i}a} \thanks{Corresponding author\newline
    }\\
  {\normalsize Department of Statistics and Computation} \\
  {\normalsize 25350 Buenavista, Saltillo, Coahuila, Mexico} \\
  {\normalsize E-mail: jadiaz@uaaan.mx} \\[2ex]
  \textbf{Ram\'on Guti\'errez J\'aimez} \\
  {\normalsize Department of Statistics and O.R.} \\
  {\normalsize University of Granada} \\
  {\normalsize Granada 18071, Spain}\\
  {\normalsize E-mail: rgjaimez@ugr.es}\\
}
\date{}
\begin{document}
\maketitle

\begin{abstract}
This paper proposes a unified approach to enable the study of diverse distributions in the
real, complex, quaternion and octonion cases, simultaneously. In particular, the central,
nonsingular matricvariate and matrix multivariate Pearson type II distribution, beta type I
distributions and the joint density of the singular values are obtained for real normed
division algebras.
\end{abstract}

\section{Introduction}\label{sec1}

In the last twenty years, the concept and the statistical and mathematical techniques known as
\emph{multivariate analysis} have changed dramatically. New statistical and mathematical tools
for the analysis of multivariate data have been developed in diverse areas of  knowledge, and
have promoted new disciplines such as pattern recognition, nonlinear multivariate analysis,
data mining, manifold learning, generalised multivariate analysis, latent variable analysis and
shape theory, among others. These and other fields constitute waht is known as \textit{modern
multivariate analysis}.

Renewed interest in multivariate analysis in the complex case has emerged in diverse areas, see
\citet{me:91}, \citet{rva:05a} and \citet{mdm:06}, among many others. Similarly, several works
involving multivariate analysis have appeared in the context of the quaternion case, see
\citet{bh:00}, \citet{f:05}, \citet{lx:09}, among others. Although receiving little attention
from a practical standpoint, but equally interesting from the theoretical point of view, some
results have appeared in the octonion case, see \citet{f:05}. This lack of widespread interest
may be, because as asserted by \citet{b:02}, \textit{...there is still no proof that the
octonions are useful for understanding the real world}. We can only hope that eventually this
question will be settled one way or another.

In terms of concepts, definitions, properties and notation from abstract algebra, it is
possible to propose a unified approach that enables the simultaneous study of the distribution
of a random matrix in real, complex, quaternion and octonion cases, which is termed as the
distribution of a random matrix for real normed division algebras.

In the real case, the matricvariate Pearson type II distribution appears in the frequentist
approach to normal regression as the distribution of the Studentised error, see
\citet{jdggj:06} and \citet{kn:04}. In Bayesian inference, the matricvariate Pearson type II
distribution is assumed as the sampling distribution; then, considering a noninformative prior
distribution, the posterior distribution and marginal distributions, the posterior mean and
generalised maximum likelihood estimators of the parameters involved are found, \citet{fl:99}.
A very important question is that of the role of the Pearson type II distribution in
multivariate analysis, because if the matrix $\mathbf{R}$ has a matricvariate Pearson type II
distribution, then the matrix $\mathbf{RR}^{*}$ (or $\mathbf{R}^{*}\mathbf{R}$) is distributed
as beta type I; and the distribution of the latter, in particular, plays a fundamental role in
the MANOVA model, see \citet{k:59,k:70} and \citet{m:82}.

The present article is organised as follows; a minimal number of concepts and the notation of
abstract algebra and Jacobians are summarised in Section \ref{sec2}. Section \ref{sec3} then
derives the nonsingular central matricvariate Pearson type II and the beta type I distributions
and some basic properties. Similarly, results are obtained for the matrix multivariate Pearson
type II and the corresponding beta type I distributions, see Section \ref{sec4}. Finally, the
joint densities of the singular values are derived in Section \ref{sec5}. We emphasise that all
these results are found for real normed division algebras.

\section{Preliminary results}\label{sec2}

A detailed discussion of real normed division algebras may be found in \citet{b:02} and
\citet{gr:87}. For convenience, we shall introduce some notation, although in general we adhere
to standard notation forms.

For our purposes, a \textbf{vector space} is always a finite-dimensional module over the field
of real numbers. An \textbf{algebra} $\mathfrak{F}$ is a vector space that is equipped with a
bilinear map $m: \mathfrak{F} \times \mathfrak{F} \rightarrow \mathfrak{F}$ termed
\emph{multiplication} and a nonzero element $1 \in \mathfrak{F}$ termed the \emph{unit} such
that $m(1,a) = m(a,1) = 1$. As usual, we abbreviate $m(a,b) = ab$ as $ab$. We do not assume
$\mathfrak{F}$ associative. Given an algebra, we freely think of real numbers as elements of
this algebra via the map $\omega \mapsto \omega 1$.

An algebra $\mathfrak{F}$ is a \textbf{division algebra} if given $a, b \in \mathfrak{F}$ with
$ab=0$, then either $a=0$ or $b=0$. Equivalently, $\mathfrak{F}$ is a division algebra if the
operation of left and right multiplications by any nonzero element is invertible. A
\textbf{normed division algebra} is an algebra $\mathfrak{F}$ that is also a normed vector
space with $||ab|| = ||a||||b||$. This implies that $\mathfrak{F}$ is a division algebra and
that $||1|| = 1$.

There are exactly four normed division algebras: real numbers ($\Re$), complex numbers
($\mathfrak{C}$), quaternions ($\mathfrak{H}$) and octonions ($\mathfrak{O}$), see
\citet{b:02}. We take into account that $\Re$, $\mathfrak{C}$, $\mathfrak{H}$ and
$\mathfrak{O}$ are the only normed division algebras; moreover, they are the only alternative
division algebras, and all division algebras have a real dimension of $1, 2, 4$ or $8$, which
is denoted by $\beta$, see \citet[Theorems 1, 2 and 3]{b:02}. In other branches of mathematics,
the parameters $\alpha = 2/\beta$ and $t = \beta/4$ are used, see \citet{er:05} and
\citet{k:84}, respectively.

Let ${\mathcal L}^{\beta}_{m,n}$ be the linear space of all $m \times n$ matrices of rank $m
\leq n$ over $\mathfrak{F}$ with $m$ distinct positive singular values, where $\mathfrak{F}$
denotes a \emph{real finite-dimensional normed division algebra}. Let $\mathfrak{F}^{m \times
n}$ be the set of all $m \times n$ matrices over $\mathfrak{F}$. The dimension of
$\mathfrak{F}^{m \times n}$ over $\Re$ is $\beta mn$. Let $\mathbf{A} \in \mathfrak{F}^{m
\times n}$, then $\mathbf{A}^{*} = \overline{\mathbf{A}}^{T}$ denotes the usual conjugate
transpose.

Table \ref{table1} sets out the equivalence between the same concepts in the four normed
division algebras.
\begin{table}[h]
  \centering
  \caption{Notation}\label{table1}
  \begin{footnotesize}
  \begin{tabular}{cccc|c}
    \hline
    Real & Complex & Quaternion & Octonion & \begin{tabular}{c}
                                               Generic \\
                                               notation \\
                                             \end{tabular}\\
    \hline
    Semi-orthogonal & Semi-unitary & Semi-symplectic & \begin{tabular}{c}
                                                         Semi-exceptional \\
                                                         type \\
                                                       \end{tabular}
      & $\mathcal{V}_{m,n}^{\beta}$ \\
    Orthogonal & Unitary & Symplectic & \begin{tabular}{c}
                                                         Exceptional \\
                                                         type \\
                                                       \end{tabular} & $\mathfrak{U}^{\beta}(m)$ \\
    Symmetric & Hermitian & \begin{tabular}{c}
                              Quaternion \\
                              hermitian \\
                            \end{tabular}
     & \begin{tabular}{c}
                              Octonion \\
                              hermitian \\
                            \end{tabular} & $\mathfrak{S}_{m}^{\beta}$ \\
    \hline
  \end{tabular}
  \end{footnotesize}
\end{table}

In addition, let $\mathfrak{P}_{m}^{\beta}$ be the \emph{cone of positive definite matrices}
$\mathbf{S} \in \mathfrak{F}^{m \times m}$; then $\mathfrak{P}_{m}^{\beta}$ is an open subset
of ${\mathfrak S}_{m}^{\beta}$.

Let $\mathfrak{D}_{m}^{\beta}$ be the \emph{diagonal subgroup} of $\mathcal{L}_{m,m}^{\beta}$
consisting of all $\mathbf{D} \in \mathfrak{F}^{m \times m}$, $\mathbf{D} = \diag(d_{1},
\dots,d_{m})$.

For any matrix $\mathbf{X} \in \mathfrak{F}^{n \times m}$, $d\mathbf{X}$ denotes the\emph{
matrix of differentials} $(dx_{ij})$. Finally, we define the \emph{measure} or volume element
$(d\mathbf{X})$ when $\mathbf{X} \in \mathfrak{F}^{m \times n}, \mathfrak{S}_{m}^{\beta}$,
$\mathfrak{D}_{m}^{\beta}$ or $\mathcal{V}_{m,n}^{\beta}$, see \citet{d:02}.

If $\mathbf{X} \in \mathfrak{F}^{m \times n}$ then $(d\mathbf{X})$ (the Lebesgue measure in
$\mathfrak{F}^{m \times n}$) denotes the exterior product of the $\beta mn$ functionally
independent variables
$$
  (d\mathbf{X}) = \bigwedge_{i = 1}^{m}\bigwedge_{j = 1}^{n}dx_{ij} \quad \mbox{ where }
    \quad dx_{ij} = \bigwedge_{k = 1}^{\beta}dx_{ij}^{(k)}.
$$

If $\mathbf{S} \in \mathfrak{S}_{m}^{\beta}$ (or $\mathbf{S} \in \mathfrak{T}_{L}^{\beta}(m)$
is a lower triangular matrix) then $(d\mathbf{S})$ (the Lebesgue measure in
$\mathfrak{S}_{m}^{\beta}$ or in $\mathfrak{T}_{L}^{\beta}(m)$) denotes the exterior product of
the $m(m+1)\beta/2$ functionally independent variables (or denotes the exterior product of the
$m(m-1)\beta/2 + n$ functionally independent variables, if $s_{ii} \in \Re$ for all $i = 1,
\dots, m$)
$$
  (d\mathbf{S}) = \left\{
                    \begin{array}{ll}
                      \displaystyle\bigwedge_{i \leq j}^{m}\bigwedge_{k = 1}^{\beta}ds_{ij}^{(k)}, &  \\
                      \displaystyle\bigwedge_{i=1}^{m} ds_{ii}\bigwedge_{i < j}^{m}\bigwedge_{k = 1}^{\beta}
                      ds_{ij}^{(k)}, & \hbox{if } s_{ii} \in \Re.
                    \end{array}
                  \right.
$$
The context generally establishes the conditions on the elements of $\mathbf{S}$, that is, if
$s_{ij} \in \Re$, $\in \mathfrak{C}$, $\in \mathfrak{H}$ or $ \in \mathfrak{O}$. It is
considered that
$$
  (d\mathbf{S}) = \bigwedge_{i \leq j}^{m}\bigwedge_{k = 1}^{\beta}ds_{ij}^{(k)}
   \equiv \bigwedge_{i=1}^{m} ds_{ii}\bigwedge_{i < j}^{m}\bigwedge_{k =
1}^{\beta}ds_{ij}^{(k)}.
$$
Observe, too, that for the Lebesgue measure $(d\mathbf{S})$ defined thus, it is required that
$\mathbf{S} \in \mathfrak{P}_{m}^{\beta}$, that is, $\mathbf{S}$ must be a non singular
Hermitian matrix (Hermitian definite positive matrix).

If $\mathbf{\Lambda} \in \mathfrak{D}_{m}^{\beta}$ then $(d\mathbf{\Lambda})$ (the Legesgue
measure in $\mathfrak{D}_{m}^{\beta}$) denotes the exterior product of the $\beta m$
functionally independent variables
$$
  (d\mathbf{\Lambda}) = \bigwedge_{i = 1}^{n}\bigwedge_{k = 1}^{\beta}d\lambda_{i}^{(k)}.
$$
If $\mathbf{H}_{1} \in \mathcal{V}_{m,n}^{\beta}$ then
$$
  (\mathbf{H}^{*}_{1}d\mathbf{H}_{1}) = \bigwedge_{i=1}^{m} \bigwedge_{j =i+1}^{n}
  \mathbf{h}_{j}^{*}d\mathbf{h}_{i}.
$$
where $\mathbf{H} = (\mathbf{H}^{*}_{1}|\mathbf{H}^{*}_{2})^{*} = (\mathbf{h}_{1}, \dots,
\mathbf{h}_{m}|\mathbf{h}_{m+1}, \dots, \mathbf{h}_{n})^{*} \in \mathfrak{U}^{\beta}(n)$. It
can be proved that this differential form does not depend on the choice of the $\mathbf{H}_{2}$
matrix. When $n = 1$; $\mathcal{V}^{\beta}_{m,1}$ defines the unit sphere in
$\mathfrak{F}^{m}$. This is, of course, an $(m-1)\beta$- dimensional surface in
$\mathfrak{F}^{m}$. When $n = m$ and denoting $\mathbf{H}_{1}$ by $\mathbf{H}$,
$(\mathbf{H}d\mathbf{H}^{*})$ is termed the \emph{Haar measure} on $\mathfrak{U}^{\beta}(m)$.

The surface area or volume of the Stiefel manifold $\mathcal{V}^{\beta}_{m,n}$ is
\begin{equation}\label{vol}
    \Vol(\mathcal{V}^{\beta}_{m,n}) = \int_{\mathbf{H}_{1} \in
  \mathcal{V}^{\beta}_{m,n}} (\mathbf{H}_{1}d\mathbf{H}^{*}_{1}) =
  \frac{2^{m}\pi^{mn\beta/2}}{\Gamma^{\beta}_{m}[n\beta/2]},
\end{equation}
where $\Gamma^{\beta}_{m}[a]$ denotes the multivariate \emph{Gamma function} for the space
$\mathfrak{S}_{m}^{\beta}$, and is defined by
\begin{eqnarray*}
  \Gamma_{m}^{\beta}[a] &=& \displaystyle\int_{\mathbf{A} \in \mathfrak{P}_{m}^{\beta}}
  \etr\{-\mathbf{A}\} |\mathbf{A}|^{a-(m-1)\beta/2 - 1}(d\mathbf{A}) \\
    &=& \pi^{m(m-1)\beta/4}\displaystyle\prod_{i=1}^{m} \Gamma[a-(i-1)\beta/2],
\end{eqnarray*}
where $\etr(\cdot) = \exp(\tr(\cdot))$, $|\cdot|$ denotes the determinant and $\re(a)
> (m-1)\beta/2$, see \citet{gr:87}. Similarly, from \citet{h:55} the \emph{multivariate beta function} for the
space $\mathfrak{S}^{\beta}_{m}$, can be defined as
\begin{eqnarray}
    \mathcal{B}_{m}^{\beta}[b,a] &=& \int_{\mathbf{0}<\mathbf{B}<\mathbf{I}_{m}}
    |\mathbf{B}|^{a-(m-1)\beta/2-1} |\mathbf{I}_{m} - \mathbf{B}|^{b-(m+1)\beta/2-1}
    (d\mathbf{B}) \nonumber\\
    &=& \int_{\mathbf{A} \in \mathfrak{P}_{m}^{\beta}} |\mathbf{A}|^{a-(m-1)\beta/2-1}
    |\mathbf{I}_{m} + \mathbf{A}|^{-(a+b)} (d\mathbf{A}) \nonumber\\ \label{beta}
    &=& \frac{\Gamma_{m}^{\beta}[a] \Gamma_{m}^{\beta}[b]}{\Gamma_{m}^{\beta}[a+b]},
\end{eqnarray}
where $\mathbf{A} = (\mathbf{I}-\mathbf{B})^{-1} -\mathbf{I}$, Re$(a) > (m-1)\beta/2$ and
Re$(b)> (m-1)\beta/2$.

Now, we show three Jacobians in terms of the $\beta$ parameter, which are based on the work of
\citet{k:84} and \citet{d:02}. These results are proposed as extensions of real, complex or
quaternion cases, see  \citet{j:64}, \citet{k:65}, \citet{me:91}, \citet{rva:05a} and
\citet{lx:09}, also see \citet{jdggj:09a}.

\begin{lem}\label{lemlt}
Let $\mathbf{X}$ and $\mathbf{Y} \in {\mathcal L}_{m,n}^{\beta}$, and let $\mathbf{Y} =
\mathbf{AXB} + \mathbf{C}$, where $\mathbf{A} \in {\mathcal L}_{m,m}^{\beta}$, $\mathbf{B} \in
{\mathcal L}_{n,n}^{\beta}$ and $\mathbf{C} \in {\mathcal L}_{m,n}^{\beta}$ are constant
matrices. Then
\begin{equation}\label{lt}
    (d\mathbf{Y}) = |\mathbf{A}^{*}\mathbf{A}|^{\beta n/2} |\mathbf{B}^{*}\mathbf{B}|^{\beta
    m/2}(d\mathbf{X}).
\end{equation}
\end{lem}

\begin{lem}[Singular value decomposition, $SVD$]\label{lemsvd}
Let $\mathbf{X} \in {\mathcal L}_{m,n}^{\beta}$, such that $\mathbf{X} =
\mathbf{V}^{*}\mathbf{DW}_{1}$ with $\mathbf{V} \in \mathfrak{U}^{\beta}(m)$, $\mathbf{W}_{1}
\in \mathcal{V}_{m,n}^{\beta}$ and $\mathbf{D} = \diag(d_{1}, \cdots,d_{m}) \in
\mathfrak{D}_{m}^{1}$, $d_{1}> \cdots > d_{m} > 0$. Then
\begin{equation}\label{svd}
    (d\mathbf{X}) = 2^{-m}\pi^{\tau} \prod_{i = 1}^{m} d_{i}^{\beta(n - m + 1) -1}
    \prod_{i < j}^{m}(d_{i}^{2} - d_{j}^{2})^{\beta} (d\mathbf{D}) (\mathbf{V}d\mathbf{V}^{*})
    (\mathbf{W}_{1}d\mathbf{W}_{1}^{*}),
\end{equation}
where
$$
  \tau = \left\{
             \begin{array}{rl}
               0, & \beta = 1; \\
               -m, & \beta = 2; \\
               -2m, & \beta = 4; \\
               -4m, & \beta = 8.
             \end{array}
           \right.
$$
\end{lem}

\begin{lem}\label{lemW}
Let $\mathbf{X} \in {\mathcal L}_{m,n}^{\beta}$, and  $\mathbf{S} = \mathbf{X}\mathbf{X}^{*}
\in \mathfrak{P}_{m}^{\beta}.$ Then
\begin{equation}\label{w}
    (d\mathbf{X}) = 2^{-m} |\mathbf{S}|^{\beta(n - m + 1)/2 - 1}
    (d\mathbf{S})(\mathbf{V}_{1}d\mathbf{V}_{1}^{*}).
\end{equation}
\end{lem}

\section{Matricvariate Pearson type II distribution}\label{sec3}

In the real case, the \emph{matricvariate Pearson type II distribution} (also known in the
literature as \textbf{\emph{matricvariate inverted $T$ distribution}}) was studied in detail by
\citet{di:67} and \citet{c:96}, see also \citet{p:82}. This distribution was previously studied
by \citet{k:59}, also in the real case.
\begin{thm}\label{teo1}
Let $\mathbf{R}\in {\mathcal L}_{m,n}^{\beta}$ defined as
$$
  \mathbf{R} = \mathbf{L}^{-1}\mathbf{X}
$$
where $\mathbf{L}$ is any square root of $\mathbf{U} = (\mathbf{U}_{1}+\mathbf{XX}^{*})$ such
that $\mathbf{LL}^{*} = \mathbf{U}$, $\mathbf{U}_{1} \sim \mathcal{W}_{m}^{\beta}(\nu,
\mbox{\textbf{\cyr I}})$, $\nu > \beta(m-1)$, independent of $\mathbf{X} \sim \mathcal{N}_{m
\times n}^{\beta}(\mathbf{0}, \mathbf{I}_{m}\otimes \mbox{\textbf{\cyr I}})$, and
$\mbox{\textbf{\cyr I}} \in \mathfrak{P}_{m}^{\beta}$. Then $\mathbf{U} \sim
\mathcal{W}_{m}^{\beta}(\nu+n, \mbox{\textbf{\cyr I}})$ independently of $\mathbf{R}$.
Furthermore, the density of $\mathbf{R}$ is
\begin{equation}\label{IT}
    \frac{\Gamma_{m}^{\beta}[\beta(n+\nu)/2]}{\pi^{mn\beta/2}\Gamma_{m}^{\beta}[\beta \nu/2]}
    |\mathbf{I}_{m} - \mathbf{RR}^{*}|^{\beta(\nu-m+1)/2-1}, \quad \mathbf{I}_{m} -
    \mathbf{RR}^{*} \in \mathfrak{P}_{m}^{\beta}
\end{equation}
which is termed the \emph{matricvariate Pearson type II distribution}\footnote{In the
literature it is customary to use the real matricvariate Pearson type II distribution, complex
matricvariate Pearson type II distribution, quaternion matricvariate Pearson type II
distribution and octonion matricvariate Pearson type II distribution,; here, however, we use
simply matricvariate Pearson type II distribution as the generic term.}.
\end{thm}
\begin{proof}
From \citet{k:84} and \citet{jdggj:09a,jdggj:10b} the joint density of $\mathbf{U}_{1}$ and
$\mathbf{X}$ is
$$
  \propto |\mathbf{U}_{1}|^{\beta(\nu-m+1)/2-1}\etr\{-\beta\mbox{\textbf{\cyr I}}^{-1}
  (\mathbf{U}_{1} + \mathbf{XX}^{*})/2\},
$$
where the constant of proportionality is
$$
  c = \frac{1}{(2\beta^{-1})^{\beta m\nu/2} \Gamma_{m}^{\beta}[\beta \nu/2] |\mbox{\textbf{\cyr I}}|^{\beta \nu/2}}
      \ \cdot \ \frac{1}{(2\pi\beta^{-1})^{\beta mn/2} |\mbox{\textbf{\cyr I}}|^{\beta n/2}}.
$$
Making the change of variable $\mathbf{U}_{1} = (\mathbf{U} -\mathbf{XX}^{*})$ and $\mathbf{X}
= \mathbf{L}\mathbf{R}$, where $\mathbf{U} = \mathbf{LL}^{*}$, then by (\ref{lt})
$$
  (d\mathbf{U}_{1})(d\mathbf{X}) = |\mathbf{LL}^{*}|^{\beta n/2}(d\mathbf{U})(d\mathbf{R})
    = |\mathbf{U}|^{\beta n/2}(d\mathbf{U})(d\mathbf{R}),
$$
and observing that $|\mathbf{U}_{1}| = |\mathbf{U} - \mathbf{XX}^{*}| = |\mathbf{U} -
\mathbf{LRR}^{*}\mathbf{L}^{*}| = |\mathbf{U}||\mathbf{I}_{m} - \mathbf{RR}^{*}|$, the joint
density of $\mathbf{U}$ and $\mathbf{R}$ is
$$
  \propto |\mathbf{I}_{m} - \mathbf{RR}^{*}|^{\beta(\nu-m+1)/2-1} |\mathbf{U}|^{\beta(\nu+n-m+1)/2-1}
  \etr\{-\beta\mbox{\textbf{\cyr I}}^{-1}\mathbf{U}/2\}. \mbox{\qed}
$$
\end{proof}
Similarly to \citet{di:67}, (\ref{IT}) is alternatively given by
\begin{equation}\label{IT2}
    \frac{\Gamma_{n}^{\beta}[\beta(n+\nu)/2]}{\pi^{mn\beta/2}\Gamma_{n}^{\beta}[\beta(n+\nu-m)/2]}
    |\mathbf{I}_{n} - \mathbf{R}^{*}\mathbf{R}|^{\beta(\nu-m+1)/2-1}, \quad \mathbf{I}_{n} -
    \mathbf{R}^{*}\mathbf{R} \in \mathfrak{P}_{n}^{\beta}.
\end{equation}
\begin{cor}\label{cor1}
Let $\mathbf{Q} = (\mathbf{M}^{*})^{-1}\mathbf{RN}^{-1} + \boldsymbol{\mu}$, $\mathbf{R}$ as in
Theorem \ref{teo1}, $\mathbf{M}$ and $\mathbf{N}$ are any square root of the constant matrices
\textbf{\cyr B}$ = \mathbf{M}\mathbf{M}^{*}\in \mathfrak{P}_{m}^{\beta}$ and \textbf{\cyr D}$=
\mathbf{N}\mathbf{N}^{*}\in \mathfrak{P}_{n}^{\beta}$, respectively, and $\boldsymbol{\mu} \in
\mathcal{L}_{m,n}^{\beta}$ is constant. Then, from (\ref{IT}) and (\ref{IT2}) the density of
$\mathbf{Q}$ is given by
$$
  \frac{\Gamma_{m}^{\beta}[\beta(n+\nu)/2]}{\pi^{mn\beta/2}\Gamma_{m}^{\beta}[\beta \nu/2]}
  \frac{|\mbox{\textbf{\cyr B}}|^{\beta(\nu+n-m+1)/2-1}}{|\mbox{\textbf{\cyr D}}|^{\beta m/2}}|\mbox{\textbf{\cyr B}}^{-1} -
  (\mathbf{Q}-\boldsymbol{\mu})\mbox{\textbf{\cyr D}}^{-1}(\mathbf{Q}-\boldsymbol{\mu})^{*}|^{\beta(\nu-m+1)/2-1},
$$
\par\noindent\hfill\mbox{$\mbox{\textbf{\cyr B}}^{-1} - (\mathbf{Q}-\boldsymbol{\mu})\mbox{\textbf{\cyr
D}}^{-1}(\mathbf{Q}-\boldsymbol{\mu})^{*} \in \mathfrak{P}_{m}^{\beta}$,}\par\noindent %
and as
$$
  \frac{\Gamma_{n}^{\beta}[\beta(n+\nu)/2]}{\pi^{mn\beta/2}\Gamma_{n}^{\beta}[\beta(\nu+n-m)/2]}
  \frac{|\mbox{\textbf{\cyr B}}|^{\beta\nu/2}}{|\mbox{\textbf{\cyr D}}|^{\beta(\nu+1)/2-1}}|\mbox{\textbf{\cyr D}} -
  (\mathbf{Q}-\boldsymbol{\mu})^{*}\mbox{\textbf{\cyr B}}(\mathbf{Q}-\boldsymbol{\mu})|^{\beta(\nu-m+1)/2-1},
$$
\par\noindent\hfill\mbox{
$\mbox{\textbf{\cyr D}} - (\mathbf{Q}-\boldsymbol{\mu})^{*}\mbox{\textbf{\cyr
B}}(\mathbf{Q}-\boldsymbol{\mu}) \in \mathfrak{P}_{n}^{\beta}$.}\par\indent\newline%
This fact shall denoted as $\mathbf{Q} \sim \mathcal{P}II_{m \times n}^{\beta}(\nu,
\boldsymbol{\mu}, \mbox{\textbf{\cyr B}},\mbox{\textbf{\cyr D}})$ (and of course $\mathbf{R}
\sim \mathcal{P}II_{m \times n}^{\beta}(\nu, \boldsymbol{0}, \mathbf{I}_{m},\mathbf{I}_{n})$).
\end{cor}
\begin{proof}
The proof follows observing that, by (\ref{lt})
$$
  (d\mathbf{R}) = |\mathbf{M}\mathbf{M}^{*}|^{\beta n/2} |\mathbf{N}\mathbf{N}^{*}|^{\beta
    m/2}(d\mathbf{Q}) = |\mbox{\textbf{\cyr B}}|^{\beta n/2} |\mbox{\textbf{\cyr D}}|^{\beta
    m/2}(d\mathbf{Q})
$$
and that
$$
  |\mbox{\textbf{\cyr B}}^{-1} - (\mathbf{Q}-\boldsymbol{\mu})\mbox{\textbf{\cyr
  D}}^{-1}(\mathbf{Q}-\boldsymbol{\mu})^{*}| = |\mbox{\textbf{\cyr B}}|^{-1}|\mbox{\textbf{\cyr D}}|^{-1}
  |\mbox{\textbf{\cyr D}} - (\mathbf{Q}-\boldsymbol{\mu})^{*} \mbox{\textbf{\cyr
  B}}(\mathbf{Q}-\boldsymbol{\mu})|. \mbox{\qed}
$$
\end{proof}
\begin{cor}\label{cor2}
Assume that $\mathbf{Q} \sim \mathcal{P}II_{m \times n}^{\beta}(\nu, \boldsymbol{\mu},
\mbox{\textbf{\cyr B}},\mbox{\textbf{\cyr D}})$, then
$$
 \mathbf{Q}^{*} \sim \mathcal{P}II_{n \times m}^{\beta}(\nu, \boldsymbol{\mu}^{*},
 \mbox{\textbf{\cyr D}}^{-1},\mbox{\textbf{\cyr B}}^{-1}).
$$
\end{cor}
\begin{proof}
The proof follows immediately from the two expressions for the density of $\mathbf{Q}$ in
Corollary \ref{cor1}. \qed
\end{proof}
\begin{cor}\label{cor3}
Let $\mathbf{R}\in {\mathcal L}_{m,n}^{\beta}$ defined as
$$
  \mathbf{R} = \mathbf{Y}\mathbf{L}_{1}^{-1}
$$
where $\mathbf{L}_{1}$ is any square root of $\mathbf{V} =
(\mathbf{V}_{1}+\mathbf{Y}^{*}\mathbf{Y})$ such that $\mathbf{L}_{1}\mathbf{L}_{1}^{*} =
\mathbf{V}$, $\mathbf{V}_{1} \sim \mathcal{W}_{n}^{\beta}(\nu+n-m, \mbox{\textbf{\cyr L}})$,
independent of $\mathbf{Y} \sim \mathcal{N}_{m \times n}^{\beta}(\mathbf{0},\mbox{\textbf{\cyr
L}} \otimes \mathbf{I}_{n})$, and $\mbox{\textbf{\cyr L}} \in \mathfrak{P}_{n}^{\beta}$. Then
$\mathbf{V} \sim \mathcal{W}_{n}^{\beta}(\nu+n, \mbox{\textbf{\cyr L}})$ independently of
$\mathbf{R}$. Furthermore, $\mathbf{R} \sim \mathcal{P}II_{m \times n}^{\beta}(\nu,
\boldsymbol{0}, \textbf{I}_{m},\textbf{I}_{n})$.
\end{cor}
\begin{proof}
The proof is a verbatim copy of the proof of Theorem \ref{teo1}. \qed
\end{proof}

Now, assume that $\mathbf{R} \sim \mathcal{P}II_{m \times n}^{\beta}(\nu, \boldsymbol{0},
\mathbf{I}_{m},\mathbf{I}_{n})$ with $n \geq m$ and let $\mathbf{B} \in
\mathfrak{P}_{m}^{\beta}$ defined as $\mathbf{B} = \mathbf{RR}^{*}$ then, under the conditions
of Theorem \ref{teo1} and Corollary \ref{cor3}, we have
\begin{eqnarray*}
  \mathbf{B} &=& \mathbf{L}^{-1}\mathbf{XX}^{*}(\mathbf{L}^{-1})^{*} = \mathbf{L}^{-1}\mathbf{W}(\mathbf{L}^{-1})^{*} \\
   &=& \mathbf{Y}(\mathbf{V}_{1}+\mathbf{Y}^{*}\mathbf{Y})^{-1}\mathbf{Y}^{*},
\end{eqnarray*}
where $\mathbf{W}=\mathbf{XX}^{*} \sim \mathcal{W}_{m}^{\beta}(n, \mbox{\textbf{\cyr I}})$, $n
> \beta(m-1)$. Thus:
\begin{thm}\label{teo2}
The density of $\mathbf{B}$ is
\begin{equation}\label{BI1}
    \frac{1}{\mathcal{B}_{m}^{\beta}[\beta \nu/2, \beta n/2]}|\mathbf{B}|^{\beta(n-m+1)/2-1}
    |\mathbf{I}_{m}-\mathbf{B}|^{\beta(\nu-m+1)/2-1},
\end{equation}
where $\mathcal{B}_{m}^{\beta}[\cdot, \cdot]$ is given by (\ref{beta}) and $\mathbf{B}$ is said
to have a \emph{matricvariate beta type I distribution}.
\end{thm}
\begin{proof}
The proof follows from (\ref{IT}) by applying (\ref{vol}) and (\ref{w}). \qed
\end{proof}

In addition, assume that $n < m$ and let  $\widetilde{\mathbf{B}} \in \mathfrak{P}_{n}^{\beta}$
defined as $\widetilde{\mathbf{B}} = \mathbf{R}^{*}\mathbf{R}$ then, under the conditions of
Theorem \ref{teo1} and Corollary \ref{cor3} we have
\begin{eqnarray*}
  \widetilde{\mathbf{B}} &=& \mathbf{X}^{*}(\mathbf{U}_{1}+\mathbf{X}^{*}\mathbf{X})^{-1}\mathbf{X}\\
   &=& \mathbf{L}_{1}^{-1}\mathbf{Y}^{*}\mathbf{Y}(\mathbf{L}_{1}^{-1})^{*} =
   \mathbf{L}_{1}^{-1}\mathbf{W}_{1}(\mathbf{L}_{1}^{-1})^{*}
\end{eqnarray*}
where $\mathbf{W}_{1} = \mathbf{Y}^{*}\mathbf{Y} \sim \mathcal{W}_{n}^{\beta}(m,
\mbox{\textbf{\cyr I}})$, $m
> \beta(n-1)$, Thus:
\begin{thm}\label{teo3}
$\widetilde{\mathbf{B}}$ has the density
\begin{equation}\label{BI11}
    \frac{1}{\mathcal{B}_{n}^{\beta}[\beta (\nu+n-m)/2, \beta m/2]}|\widetilde{\mathbf{B}}|^{\beta(m-n+1)/2-1}
    |\mathbf{I}_{n}-\widetilde{\mathbf{B}}|^{\beta(\nu-m+1)/2-1}.
\end{equation}
Also, we say that $\widetilde{\mathbf{B}}$ has a \emph{matricvariate beta type I distribution}.
\end{thm}
\begin{proof}
The proof is the same as that given in Theorem \ref{teo2}. Alternatively, observe that density
(\ref{BI11}) can be obtained from density (\ref{BI1}) making the following substitutions, see
\citet[Eq. (7), p. 455]{m:82} and \citet[p. 96]{sk:79},
\begin{equation}\label{s}
    m \rightarrow n, \quad n \rightarrow m, \quad \nu \rightarrow \nu+n-m. \qquad \mbox{\qed}
\end{equation}
\end{proof}

Densities (\ref{BI1}) and (\ref{BI11}) have been studied by several authors in the real case,
see \citet{k:70} and \citet{sk:79}, \citet{c:96}, among many others; and by \citet{j:64},
\citet{jdggj:09b} and \citet{jdggj:08} and \citet{jdggj:10a}, in noncentral, doubly noncentral,
singular and nonsingular and complex cases, among many other authors. By the analogous
construction of $\mathbf{B} = \mathbf{RR}^{*}$ (or $\widetilde{\mathbf{B}} =
\mathbf{R}^{*}\mathbf{R}$) in terms of $\mathbf{R}$ compared with the construction of the
Wishart matrix in terms of matrix multivariate normal distribution, the distribution of
$\mathbf{B} = \mathbf{RR}^{*}$ (or $\widetilde{\mathbf{B}} = \mathbf{R}^{*}\mathbf{R}$) is
sometimes termed the studentised Wishart distribution. \citet{k:84} studied densities
(\ref{BI1}) and (\ref{BI11}) for the hypercomplex case.

\section{Matrix multivariate Pearson type II distribution}\label{sec4}

\begin{thm}\label{teo4}
Let $(S^{1/2})^{2} \sim \chi^{2, \beta}(\nu)$ independent of $\mathbf{Y} \sim \mathcal{N}_{m
\times n}^{\beta}(\mathbf{0}, \mathbf{I}_{m} \otimes \mathbf{I}_{n})$, and define $S_{1} = S +
\tr \mathbf{YY}*$ and $\mathbf{R}_{1} = S_{1}^{-1/2}\mathbf{Y}$. Then $S_{1} \sim \chi^{2}(\nu
+ mn)$ independently of $\mathbf{R}_{1}$. Furthermore, the density of $\mathbf{R}_{1}$ is
\begin{equation}\label{mmp}
    \frac{\Gamma^{\beta}_{1}[\beta(\nu+mn)/2]}{\pi^{\beta mn/2}\Gamma^{\beta}_{1}[\beta \nu/2]}
    (1-\tr \mathbf{R}_{1}\mathbf{R}_{1}^{*})^{\beta \nu/2-1}, \quad 1-\tr
    \mathbf{R}_{1}\mathbf{R}_{1}^{*} > 0,
\end{equation}
which is termed the \emph{matrix multivariate Pearson type II distribution}.
\end{thm}
\begin{proof}
The joint density of $S$ and $\mathbf{Y}$ is
$$
  \propto s^{\beta \nu/2 -1} \etr\{-\beta (s + \tr \mathbf{YY}^{*})/2\},
$$
and the desired results are obtained analogously to the proof of Theorem \ref{teo1}. \qed
\end{proof}
\begin{cor}
Let $\mathbf{Q}_{1} = (\mathbf{M}^{*})^{-1}\mathbf{R}_{1}\mathbf{N}^{-1} + \boldsymbol{\mu}$,
$\mathbf{R}_{1}$ as in Theorem \ref{teo3}, $\mathbf{M}$ and $\mathbf{N}$ are any square root of
the constant matrices \textbf{\cyr B}$ = \mathbf{M}\mathbf{M}^{*}\in \mathfrak{P}_{m}^{\beta}$
and \textbf{\cyr D}$= \mathbf{N}\mathbf{N}^{*}\in \mathfrak{P}_{n}^{\beta}$, respectively, and
$\boldsymbol{\mu} \in \mathcal{L}_{m,n}^{\beta}$ is constant. Then,
$$
    \frac{\Gamma^{\beta}_{1}[\beta(\nu+mn)/2]}{\pi^{\beta mn/2}\Gamma^{\beta}_{1}[\beta \nu/2]}
    |\mbox{\textbf{\cyr B}}|^{\beta n/2} |\mbox{\textbf{\cyr D}}|^{\beta
    m/2}\left[1-\tr \mbox{\textbf{\cyr B}}(\mathbf{Q}_{1} - \boldsymbol{\mu})\mbox{\textbf{\cyr D}}
    (\mathbf{Q}_{1} - \boldsymbol{\mu})^{*}\right]^{\beta \nu/2-1},
$$
\par\noindent\hfill\mbox{$1-\tr \mbox{\textbf{\cyr B}}(\mathbf{Q}_{1} - \boldsymbol{\mu})\mbox{\textbf{\cyr D}}
    (\mathbf{Q}_{1} - \boldsymbol{\mu})^{*} > 0$.}\par\noindent %
Hence, we write
$$
  \mathbf{Q}_{1} \sim \mathcal{MP}II_{m \times n}^{\beta}(\nu,
  \boldsymbol{\mu}, \mbox{\textbf{\cyr B}}, \mbox{\textbf{\cyr D}}),
$$
and, in particular $\mathbf{R}_{1} \sim \mathcal{MP}II_{m \times n}^{\beta}(\nu,
\boldsymbol{0}, \mathbf{I}_{m}, \mathbf{I}_{n})$.
\end{cor}
\begin{proof}
The proof follows observing that, by (\ref{lt})
$$
  (d\mathbf{R}_{1}) = |\mathbf{M}\mathbf{M}^{*}|^{\beta n/2} |\mathbf{N}\mathbf{N}^{*}|^{\beta
    m/2}(d\mathbf{Q}_{1}) = |\mbox{\textbf{\cyr B}}|^{\beta n/2} |\mbox{\textbf{\cyr D}}|^{\beta
    m/2}(d\mathbf{Q}_{1}). \qquad \mbox{\qed}
$$
\end{proof}
Now, assuming that $\mathbf{R}_{1} \sim \mathcal{P}II_{m \times n}^{\beta}(\nu, \boldsymbol{0},
\mathbf{I}_{m},\mathbf{I}_{n}),$ with $n \geq m$ and defining $\mathbf{B}_{1} =
\mathbf{R}_{1}\mathbf{R}_{1}^{*} \in \mathfrak{P}_{m}^{\beta}$, then, under the conditions of
Theorem \ref{teo4} we have that
$$
  \mathbf{B}_{1} = S^{-1}\mathbf{YY}^{*}
$$
where $\mathbf{W}=\mathbf{YY}^{*} \sim \mathcal{W}_{m}^{\beta}(n, \mathbf{I}_{m})$, $n
> \beta(m-1)$. Furthermore:
\begin{thm}\label{teo5}
The density of $\mathbf{B}_{1}$ is
\begin{equation}\label{BI2}
    \frac{\Gamma^{\beta}_{1}[\beta(\nu+mn)/2]}{\Gamma^{\beta}_{1}[\beta \nu/2] \Gamma_{m}^{\beta}[\beta n/2]}
    |\mathbf{B}_{1}|^{\beta(n-m+1)/2-1}(1-\tr \mathbf{B}_{1})^{\beta \nu/2-1}, \quad \mathbf{0} < \mathbf{B}_{1} < \mathbf{I}_{m},
\end{equation}$\mathbf{B}_{1}$ is said to have a \emph{matrix multivariate beta type I distribution}.
\end{thm}
\begin{proof}
The proof follows from (\ref{mmp}) by applying (\ref{vol}) and (\ref{w}). \qed
\end{proof}
Similarly, if $n<m$ and $\widetilde{\mathbf{B}}_{1} = \mathbf{R}_{1}^{*}\mathbf{R}_{1} \in
\mathfrak{P}_{n}^{\beta}$, thus:
\begin{thm}\label{teo6}
$\widetilde{\mathbf{B}}_{1}$ has the density
\begin{equation}\label{BI21}
    \frac{\Gamma^{\beta}_{1}[\beta(\nu+mn)/2]}{\Gamma^{\beta}_{1}[\beta \nu/2] \Gamma_{n}^{\beta}[\beta m/2]}
    |\widetilde{\mathbf{B}}_{1}|^{\beta(m-n+1)/2-1}(1-\tr \widetilde{\mathbf{B}}_{1})^{\beta \nu/2-1},
    \quad \mathbf{0} < \widetilde{\mathbf{B}}_{1} < \mathbf{I}_{n}.
\end{equation}
Thus, $\widetilde{\mathbf{B}}_{1}$ is said to have a \emph{matrix multivariate distribution
type I distribution}.
\end{thm}
\begin{proof}
Density (\ref{BI21}) can be obtained from density (\ref{BI2}) by making the following
substitutions,
\begin{equation}\label{s2}
    m \rightarrow n, \quad n \rightarrow m. \qquad \qquad \mbox{\qed}
\end{equation}
\end{proof}

\section{Singular value densities}\label{sec5}
In this section, we derive the joint density of the singular values of matrices $\mathbf{R}$,
$\widetilde{\mathbf{R}}$, $\mathbf{R}_{1}$ and $\widetilde{\mathbf{R}}_{1}$. In addition, the
joint densities of the eigenvalues of $\mathbf{B}$, $\widetilde{\mathbf{B}}$, $\mathbf{B}_{1}$
and $\widetilde{\mathbf{B}}_{1}$ are obtained.

\begin{thm}\label{teosv}
Let $\delta_{1}, \dots, \delta_{m}$ be the singular values of $\mathbf{R} \sim \mathcal{P}II_{m
\times n}^{\beta}(\nu,\mathbf{0}, \mathbf{I}_{m}, \mathbf{I}_{n})$, $1 > \delta_{1}> \cdots >
\delta_{m} > 0$. Then its joint density is
\begin{equation}\label{svPII}
    \frac{2^{m} \ \pi^{\beta m^{2}+\tau}}{\Gamma_{m}^{\beta}[\beta m/2] \mathcal{B}_{m}^{\beta}[\beta \nu/2, \beta n/2]}
    \prod_{i=1}^{m} \delta_{i}^{\beta(n-m+1)-1}(1-\delta_{i}^{2})^{\beta(\nu-m+1)/2-1}
    \prod_{i<j}^{m}(\delta_{i}^{2} - \delta_{j}^{2})^{\beta}
\end{equation}
where $\tau$ is defined in Lemma \ref{lemsvd}.
\end{thm}
\begin{proof}
The proof follows immediately from (\ref{IT}), first using (\ref{svd}) and then applying
(\ref{vol}). \qed
\end{proof}
The joint density of the singular values of $\widetilde{\mathbf{R}}$ is obtained from
(\ref{svPII}) after making the substitutions (\ref{s}).

\begin{thm}
Suppose that $\mathbf{R}_{1} \sim \mathcal{MP}II_{m \times n}^{\beta}(\nu,\mathbf{0},
\mathbf{I}_{m}, \mathbf{I}_{n})$ and let $\alpha_{1}, \dots, \alpha_{m}$, $1 > \alpha_{1}>
\cdots >, \alpha_{m} > 0$, $0 < \sum_{i=1}^{m}\alpha_{i} < 1$, its singular values. Then its
joint density is
\begin{equation}\label{svMPII}
  \frac{2^{m} \pi^{\beta m^{2}/2 + \tau}\Gamma_{1}^{\beta}[\beta(\nu +
  mn)/2]}{\Gamma_{1}^{\beta}[\beta \nu/2]
  \Gamma_{m}^{\beta}[\beta m/2]\Gamma_{m}^{\beta}[\beta n/2]}
  \left(1-\sum_{i=1}^{m}\alpha_{i}^{2}\right)^{\beta \nu/2-1}\prod_{i=1}^{m} \alpha_{i}^{\beta(n-m+1)-1}
  \prod_{i<j}^{m}(\alpha_{i}^{2} - \alpha_{j}^{2})^{\beta}
\end{equation}
\end{thm}
\begin{proof}
The proof is analogous to that given for Theorem \ref{teosv}. \qed
\end{proof}
Similarly, the joint density of the singular values of $\widetilde{\mathbf{R}}_{1}$ is obtained
from (\ref{svMPII}) and making the substitutions (\ref{s2}).

Finally, observe that $\delta_{i} = \sqrt{\eig_{i}(\mathbf{RR}^{*})}$  and $\alpha_{i} =
\sqrt{\eig_{i}(\mathbf{R}_{1}\mathbf{R}_{1}^{*})}$, where $\eig_{i}(\mathbf{A})$, $i = 1,
\dots, m$, denotes the $i$-th eigenvalue of $\mathbf{A}$. Let $\lambda_{i} =
\eig_{i}(\mathbf{RR}^{*})$ and $\gamma_{i} = \eig_{i}(\mathbf{R}_{1}\mathbf{R}_{1}^{*})$,
hence, observing that, for example, $\delta_{i} = \sqrt{\lambda_{i}}$ and then
$$
  \bigwedge_{i=1}^{m} d\delta_{i} = \bigwedge_{i=1}^{m} 2^{-m} \prod_{i=1}^{m}
  \lambda_{i}^{-1/2}d\lambda_{i},
$$
the corresponding joint density of $\lambda_{1}, \dots, \lambda_{m}$, $1 > \lambda_{1} > \cdots
> \lambda_{m} > 0$ is obtained from (\ref{svPII}) as
$$
    \frac{\pi^{\beta m^{2}+\tau}}{\Gamma_{m}^{\beta}[\beta m/2] \mathcal{B}_{m}^{\beta}[\beta \nu/2, \beta n/2]}
    \prod_{i=1}^{m} \lambda_{i}^{\beta(n-m+1)/2-1}(1-\lambda_{i})^{\beta(\nu-m+1)/2-1}
    \prod_{i<j}^{m}(\lambda_{i} - \lambda_{j})^{\beta}.
$$
Similarly, the joint density of $\gamma_{1}, \dots, \gamma_{m}$, $1 > \gamma_{1} > \cdots >
\gamma_{m}> 0$, $0 < \sum_{i=1}^{m}\alpha_{i} < 1$, is obtained from (\ref{svMPII}) as
$$
  \frac{\pi^{\beta m^{2}/2 + \tau}\Gamma_{1}^{\beta}[\beta(\nu +
  mn)/2]}{\Gamma_{1}^{\beta}[\beta \nu/2]
  \Gamma_{m}^{\beta}[\beta m/2]\Gamma_{m}^{\beta}[\beta n/2]}
  \left(1-\sum_{i=1}^{m}\gamma_{i}\right)^{\beta \nu/2-1}\prod_{i=1}^{m} \gamma_{i}^{\beta(n-m+1)/2-1}
  \prod_{i<j}^{m}(\gamma_{i} - \gamma_{j})^{\beta}.
$$

\section*{Conclusions}
\begin{itemize}
  \item Beyond a doubt, in any generalisation of results there is a price to be paid, and in this case the price is that of acquiring
a basic understanding of some concepts of abstract algebra, which can initially be summarised
as the use of notation and a basic minimum set of definitions. However, we believe that a
detailed study of mathematical properties from a statistical standpoint can have a potential
impact on statistical theory.
  \item Furthermore, note that $\mathbf{X}\in \mathfrak{L}^{\beta}_{m,n}$ has a matrix multivariate elliptically
contoured distribution for real normed division algebras if its density, with respect to the
Lebesgue measure, is given by (see \citet{jdggj:09a}):
\begin{equation}\label{mve}
  \frac{C^{\beta}(m,n)}{|\mathbf{\Sigma}|^{\beta n/2}|\mathbf{\Theta}|^{\beta m/2}}
  h\left\{\tr\left[\mathbf{\Sigma}^{-1}(\mathbf{X}-\boldsymbol{\mu})\mathbf{\Theta}^{-1}
  (\mathbf{X}- \boldsymbol{\mu})^{*}\right]\right\},
\end{equation}
where  $\boldsymbol{\mu}\in \mathfrak{L}^{\beta}_{m,n}$, $ \mathbf{\Sigma}\in
\mathfrak{P}^{\beta}_{m}$,  $ \mathbf{\Theta}\in \mathfrak{P}^{\beta}_{m}$. The function $h:
\mathfrak{F} \rightarrow [0,\infty)$ is termed the generator function, and it is such that
$\int_{\mathfrak{P}^{\beta}_{1}} u^{\beta nm-1}h(u^2)du < \infty$ and
$$
  C^{\beta}(m,n) = \frac{\Gamma[\beta mn/2]}{2 \pi^{\beta mn/2}} \left\{
    \int_{\mathfrak{P}^{\beta}_{1}} u^{\beta nm-1}h(u^2)du\right \}
$$
Such a distribution is denoted by $\mathbf{X}\sim \mathcal{E}^{\beta}_{n\times
m}(\boldsymbol{\mu},\mathbf{\Sigma}, \mathbf{\Theta}, h)$, for the real case, see \citet{fz:90}
and \citet{gv:93}; and \citet{mdm:06} for the complex case. Observe that this class of matrix
multivariate distributions includes normal, contaminated normal, Pearson type II and VII, Kotz,
Jensen-Logistic, power exponential and Bessel distributions, among others; these distributions
have tails that are more or less weighted, and/or present a greater or smaller degree of
kurtosis than the normal distribution.

Assume that $\mathbf{X} = (\build{\mathbf{X}_{1}}{}{m \times n}\vdots\build{\mathbf{X}_{2}}{}{m
\times \nu}) \sim \mathcal{E}^{\beta}_{m \times n+\nu}(\boldsymbol{0},\mathbf{I}_{m},
\mathbf{I}_{n+\nu}, h)$, $n,\nu \geq m$; and define, $\mathbf{R} =
\mathbf{L}^{-1}\mathbf{X}_{1}$, where $\mathbf{L}$ is any square root of $\mathbf{V} =
(\mathbf{X}_{2}\mathbf{X}_{2}^{*} + \mathbf{X}_{1}\mathbf{X}_{1}^{*})$ such that
$\mathbf{LL}^{*} = \mathbf{V}$. Then $\mathbf{R} \sim \mathcal{P}II_{m \times
n}^{\beta}(\nu,\mathbf{0}, \mathbf{I}_{m},\mathbf{I}_{n})$ independently of $\mathbf{V} \sim
\mathcal{GW}_{m}^{\beta}(n+\nu, \mathbf{I}_{m},h)$, $n+\nu > \beta(m-1)$, where
$\mathcal{GW}_{m}^{\beta}(\cdot)$ denotes the generalised Wishart distribution based on an
elliptical distribution, see \citet{jdggj:09a} and \citet{jdggj:10b}. From (\ref{mve}) the
density of $\mathbf{X}$ is
$$
  C^{\beta}(m,n+\nu) h\left\{\tr\left(\mathbf{X}_{1}\mathbf{X}_{1}^{*} +
  \mathbf{X}_{2}\mathbf{X}_{2}^{*}\right)\right\}.
$$
Let $\mathbf{V}_{0} = \mathbf{X}_{2}\mathbf{X}_{2}^{*}$ then by (\ref{lemW}),
$(d\mathbf{X}_{2}) = 2^{-m}|\mathbf{V}_{0}|^{\beta(\nu -m+1)/2-1}(d\mathbf{V}_{0})
(\mathbf{H}_{1}d\mathbf{H}_{1}^{*})$. Thus, the marginal density of $\mathbf{X}_{1}$ and
$\mathbf{V}_{0}$ is obtained by integrating over $\mathbf{H}_{1} \in
\mathcal{V}_{m,n}^{\beta}$, and then using (\ref{vol}), to obtain
$$
  \frac{C^{\beta}(m,n+\nu) \pi^{\beta \nu m/2}}{\Gamma_{m}^{\beta}[\beta \nu/2]} |\mathbf{V}_{0}|^{\beta(\nu -m+1)/2-1}
  h\left\{\tr\left(\mathbf{X}_{1}\mathbf{X}_{1}^{*} +  \mathbf{V}_{0}\right)\right\}.
$$
Now, let $\mathbf{V} = (\mathbf{V}_{0} + \mathbf{X}_{1}\mathbf{X}_{1}^{*})$ and $\mathbf{R} =
\mathbf{L}^{-1}\mathbf{X}_{1}$, where $\mathbf{L}\mathbf{L}^{*} =\mathbf{V}$, then by
(\ref{lemlt})
$$
  (d\mathbf{X}_{1})(d\mathbf{V}_{0}) =  |\mathbf{V}|^{\beta n/2}(d\mathbf{R})(d\mathbf{V}).
$$
Observing that $|\mathbf{V}_{0}| = |\mathbf{V}||\mathbf{I}_{m}-\mathbf{R}\mathbf{R}^{*}|$, and
so the joint density of $\mathbf{R}$ and $\mathbf{V}$ is
$$
  \frac{C^{\beta}(m,n+\nu) \pi^{\beta \nu m/2}}{\Gamma_{m}^{\beta}[\beta \nu/2]}
  |\mathbf{I}_{m}-\mathbf{R}\mathbf{R}^{*}|^{\beta(\nu -m+1)/2-1} |\mathbf{V}|^{\beta(n+\nu -m+1)/2-1}
  h\left\{\tr\mathbf{V}\right\}.
$$
from where the desired result follows. \qed

Observe that in this case, $\mathbf{X}_{1}$ and $\mathbf{X}_{2}$ (or $\mathbf{V}_{0} =
\mathbf{X}_{2}\mathbf{X}_{2}^{*}$) are stochastically dependent. Furthermore, note that only
when the particular matrix multivariate elliptical distribution is the matrix multivariate
normal distribution, are $\mathbf{X}_{1}$ and $\mathbf{X}_{2}$ (or $\mathbf{V}_{0} =
\mathbf{X}_{2}\mathbf{X}_{2}^{*}$) independent. Then, we can say that the matricvariate Pearson
type II distribution is invariant under the family of matrix multivariate elliptical
distributions, and furthermore, its density is the same as when normality is assumed. In the
same way, it can be proved that the matrix multivariate Pearson type II, matricvariate and
matrix multivariate beta type I distributions are invariant under the family of matrix
multivariate elliptical distributions.
\item Finally, following \citet{k:84}, the distributions studied in this paper are easily extended to
the hypercomplex case (biquaternion and bioctonion cases, which are a Jordan algebras), simply
replacing $\beta$ by $2\beta$ in the obtained results.
\end{itemize}

\section*{Acknowledgements}
This research work was partially supported  by CONACYT-M\'exico, Research Grant No. \ 81512 and
IDI-Spain, Grants No. FQM2006-2271 and MTM2008-05785. This paper was written during J. A.
D\'{\i}az-Garc\'{\i}a's stay as a visiting professor at the Department of Statistics and O. R.
of the University of Granada, Spain.

\bibliographystyle{plain}

\end{document}